\documentclass{amsart} 

\usepackage{amsfonts}
\usepackage{amsfonts}
\usepackage{latexsym}
\usepackage{amssymb}
\usepackage{amscd}
\usepackage[latin1]{inputenc}
\usepackage{verbatim}

\newtheorem{theorem}{Theorem}[section]

\newtheorem{corollary}[theorem]{Corollary}

\theoremstyle{definition}
\newtheorem{definition}{Definition}[section]

\theoremstyle{remark}

\newcommand\mn{{\rm M}_n}

\newcommand\mpee{{\rm M}_p}
\newcommand\mpeen{{\rm M}_{pn}}
\newcommand\mr{{\rm M}_r}

\newcommand\mtwopee{{\rm M}_{2p}}

\newcommand\un{{\rm U}_n}                

\newcommand\ose{{\mathcal E}}

\newcommand\osr{{\mathcal R}}
\newcommand\oss{{\mathcal S}}
\newcommand\ost{{\mathcal T}}

\begin{document}

\title{Arveson's Criterion for Unitary Similarity}
\author{Douglas Farenick}
\address{Department of Mathematics and Statistics,
University of Regina,
Regina, Saskatchewan S4S 0A2, Canada}


\maketitle

\begin{abstract}
This paper is an exposition of W.B.~Arveson's complete invariant for the
unitary similarity of complex, irreducible matrices.
\end{abstract}
  

\section*{Introduction}
\label{S:Intro}

Forty years ago W.B.~Arveson announced an important theorem
concerning the unitary similarity problem \cite{arveson1970}.
His proof of the theorem appeared two years later as a consequence of a deep study
\cite{arveson1969, arveson1972} that profoundly influenced
the subsequent development of operator algebra theory.
With the richness of the operator-algebraic results in these seminal papers,
Arveson's significant and novel contribution to linear algebra has been somewhat overshadowed.
Therefore, my aims with this exposition 
are to draw attention again to this remarkable result and
to give a self-contained proof of it. 

The method of proof is
different from Arveson's (and from Davidson's treatment \cite{davidson1981} of Arveson's approach),
and so may be considered new. However, the arguments
draw upon known results, adapted to the setting, language, and notation of linear algebra. The significant ideas
are due to other mathematicians; I have merely reconfigured them in a package
accessible to readers with a background in core linear algebra. 

The paper is intended to be self-contained. Results that have found their way into textbooks
are merely recalled for the reader's benefit. The standard references used here are
the books of Horn and Johnson \cite{Horn--Johnson-book} (for linear algebraic analysis)
and Paulsen \cite{Paulsen-book} (for completely positive linear transformations of matrix spaces).
I provide proofs for
results that may be well known (Dunford's Ergodic Theorem \cite{dunford1943},
Kadison's Isometry Theorem \cite{kadison1951}), but are not in standard textbooks. In such cases, 
the proofs treat the problem at hand rather than the
most general situation. 

We shall use the following terminology and notation.
The set of $n\times n$ matrices over the field $\mathbb C$ of complex
numbers is denoted by $\mn$, and for every $X\in\mn$ the conjugate transpose of $X$ is denoted by $X^*$. 
A matrix $X\in\mn$ is: \emph{hermitian}
if $X^*=X$; \emph{positive semidefinite} if $X=Y^*Y$ for some $Y\in\mn$; 
 \emph{unitary} if $X$ is invertible and $X^{-1}=X^*$. 
The spectral (or operator) norm of $X\in\mn$ is 
given by
\[
\|X\|\,=\,\sqrt{\mbox{\rm spr}\,(X^*X)}\,,
\]
where $\mbox{spr}\,(Y)$ denotes the spectral radius of $Y\in \mn$.
The \emph{closed unit ball} of $\mn$ is the set
\[
\mbox{\rm Ball}\left(\mn\right)\,=\,\{X\in\mn\,:\, \|X\|\leq 1\}\,,
\]
which is a convex set whose set of extreme points is $\un$ \cite{kadison1951}, \cite[\S3.1, Problem 27]{Horn--Johnson-book}.
In the metric topology of $\mn$ induced by the spectral norm, the sets $\mbox{\rm Ball}\left(\mn\right)$ and $\un$
are compact.

\section{The Unitary Similarity Problem}

Two matrices $A,B\in\mn$ are said to be \emph{unitarily similar}  
if $B=U^*AU$ for some $U\in\un$.

\begin{definition} Let $\mathcal O\subseteq\mn$ be  fixed, nonempty
subset of matrices. The \emph{unitary similarity problem for $\mathcal O$} is to find 
a countable family $\mathcal F_{\mathcal O}$ of functions
defined on $\mathcal O$ with the following two properties:
\begin{enumerate}
\item\label{invar} $f(U^*AU)=f(A)$, for all $A\in \mathcal O$, $U\in\un$, $f\in\mathcal F_{\mathcal O}$;
\item\label{complete invar} $f(A)=f(B)$, for fixed $A,B\in\mathcal O$ and for all $f\in\mathcal F_{\mathcal O}$, if and only if $B=U^*AU$ for some $U\in\un$.
\end{enumerate}
\end{definition} 

Condition \eqref{invar} above asserts that the functions $f\in\mathcal F_{\mathcal O}$ are invariant under unitary similarity and condition \eqref{complete invar}
says that these invariants are complete in the sense that if matrices $A,B\in\mathcal O$ are not unitarily equivalent, then $f(A)\neq f(B)$ for at least
one of the invariants $f\in\mathcal F_{\mathcal O}$.

In the best of circumstances, the set $\mathcal O$ is $\mn$, but that is not always to be the case, and instead one may
require that the set $\mathcal O$ be an algebraic variety or possess some good topological properties.
The set $\mathcal O$ considered by Arveson is of the latter type: it has the topology
of a second countable complete metric space.

Although now twenty years old, the survey paper by Shapiro \cite{shapiro1991} remains
a good reference for an overview of the unitary similarity problem. Perhaps the most celebrated of
all contributions to the problem are two classical results: Specht's trace invariants \cite{specht1940} and 
Littlewood's algorithm \cite{littlewood1953}.

\section{Statement of Arveson's Theorem}

\begin{definition}  Assume $X,P\in\mn$.
\begin{enumerate}
\item $P$ is a \emph{projection} if $P^*=P$ and $P^2=P$.
\item $X\in\mn$ is \emph{irreducible} if $XP=PX$, for a projection $P$, holds only if $P\in\{0,I\}$,
where $I\in\mn$ denotes the identity matrix. 
\item  $\mathcal O_{\rm irr}$ denotes the set of all irreducible matrices in $\mn$.
\end{enumerate}
\end{definition}

Equivalently, $X\in\mn$ is irreducible if and only if the algebra generated by the set $\{I,X,X^*\}$ is $\mn$.
The set  $\mathcal O_{\rm irr}$ is a dense $G_\delta$-set
\cite{halmos1968}. Therefore, $\mathcal O_{\rm irr}$ is a Polish space,
which is to say that (in the relative topology) $\mathcal O_{\rm irr}$ is a second countable complete metric space.

The set $\mathcal S$ of pairs 
$(H,K)$ of $n\times n$ matrices with entries in $\mathbb Q+i\mathbb Q$ is countable and dense in $\mn\times\mn$. Let $\mathcal F_{\mathcal O_{\rm irr}}$ be
the family of functions $f_{(H,K)}$, $(H,K)\in\mathcal S$, defined on $\mn$ by
\[
f_{(H,K)}(A)\;=\;\left\| A\otimes H+I\otimes K\right\|\,,\;A\in\mn\,.
\]
Because $U\otimes I\in {\rm U}_{n^2}$ (the unitary group of $\mn\otimes\mn$) 
for all $U\in\un$, it is clear that $f_{(H,K)}(U^*AU)=f(A)$
for all $U\in\un$ and $A\in\mn$. Hence, $\mathcal F_{\mathcal O_{\rm irr}}$ is a countable family of 
unitary similarity invariants for $\mn$. The following theorem shows, using the fact that $\mathcal S$ is dense in $\mn\times\mn$, 
that $\mathcal F_{\mathcal O_{\rm irr}}$ is a complete invariant for unitary similarity for the class $\mathcal O_{\rm irr}$.

\begin{theorem}\label{main result} {\rm (Arveson)} The following statements are equivalent for $A,B\in \mn$ such that $A\in\mathcal O_{\rm irr}$:
\begin{enumerate}
\item[{\rm (i)}]\label{cond1} $\|A\otimes H + I\otimes K\|\,=\,\|B\otimes H + I\otimes K\|$, for all $H,K\in \mn$;
\item[{\rm (ii)}]\label{cond2} $B=U^*AU$ for some $U\in\un$.
\end{enumerate}
\end{theorem}

Note that if neither $A$ nor $B$ is assumed to be irreducible, then {\rm (i)} does not imply {\rm (ii)}.
In particular, if $X$ is any irreducible matrix and if $A=X\oplus X$ and $B= X\oplus 0$, then $A$ and $B$
satisfy {\rm (i)} but not {\rm (ii)}.

The key steps in the proof of Theorem \ref{main result} are:
\begin{enumerate}
\item to show that there are unital completely positive linear transformations $\phi,\psi:\mn\rightarrow\mn$
such that $\phi(A)=B$ and $\psi(B)=A$;
\item to show that, for the transformation $\omega=\psi\circ\phi$ on $\mn$, the condition $\omega(A)=A$ implies
that $\omega(X)=X$ for every $X\in\mn$ (this is the heart of the argument and is called the \emph{Boundary Theorem});
\item to show that if a unital completely positive linear transformation of $\mn$ is an isometry, then it must be a unitary
similarity transformation (this result is known as \emph{Kadison's Isometry Theorem});
\item to use $X=\psi(\phi(X))$ for all $X\in\mn$ to show that $\phi$ is an isometry and, hence, a unitary
similarity transformation.
\end{enumerate}

\section{Completely Positive Linear Transformations of Matrix Spaces}

For a fixed $n\in\mathbb N$, our interest is with linear transformations 
$\phi:\mn\rightarrow\mn$ that leave certain matrix cones invariant,
not just at the level of $\mn$ itself, but at the level of all matrix rings over $\mn$.
 
\begin{definition} {\rm (Two Identifications of Matrix Spaces)}
Fix $n\in\mathbb N$. For every $p\in\mathbb N$ the ring $\mpeen$ of $pn\times pn$ matrices
is considered in the following two equivalent ways:
\begin{enumerate} 
\item as block matrices---namely $\mpeen=\mpee(\mn)$, the ring of $p\times p$ matrices over the ring $\mn$;
\item as tensor (Kronecker) products---that is, $\mpeen=\mn\otimes\mpee$.
\end{enumerate}
The identity matrix of $\mpee(\mn)$ is denoted by $I_n\otimes I_p$.
Likewise, if $\ost\subseteq\mn$
is any subspace, then $\mpee(\ost)$ denotes the vector space
of all $p\times p$ matrices with entries from
$\ost$ and is identified with $\ost\otimes\mpee$.
\end{definition}

\begin{definition}  {\rm (Matricial Cones and Orderings)} If $\osr\subseteq\mn$ is a subspace of matrices with the properties
\begin{enumerate}
\item $I\in\osr$ and
\item $X^*\in \osr$ for every $X\in\osr$,
\end{enumerate}
then the \emph{canonical matricial cones of $\osr$} are the sets
\[
\mpee(\osr)_+\,=\,\{H\in\mpee(\osr)\,:\,H\;\mbox{is a positive semidefinite matrix}\}\,.
\]
If $\mpee(\osr)_{\rm sa}$ denotes the real vector space of hermitian matrices of $\mpee(\osr)$
and if $X,Y\in \mpee(\osr)_{\rm sa}$, then $X\leq Y$ denotes $Y-X\in\mpee(\osr)_+$; this is called the
\emph{canonical matricial ordering of $\osr$}.
\end{definition}

The matricial cones of $\osr$ have extremely good cone-theoretic properties.
First, the set $\mpee(\osr)_+$ is a cone in the usual sense of being closed under multiplication by positive scalars and
finite sums. Moreover: this cone is \emph{pointed}, which is to say that
$\mpee(\osr)_+\cap\left(-\mpee(\osr)_+\right)=\{0\}$; it is \emph{reproducing}
in that  $\mpee(\osr)_{\rm sa}$ is obtained by taking all differences $H-K$, for $H,K\in\mpee(\osr)_+$; and
it is \emph{closed}
in the topology of $\mpee(\mn)$. Such a cone is said to be \emph{proper}.
Since $\mpee(\osr)=\mpee(\osr)_{\rm sa}+i\mpee(\osr)_{\rm sa}$, the cone $\mpee(\osr)_+$
spans $\mpee(\osr)$. 

The identity matrix of $\mpee(\mn)$ is an \emph{Archimedean order unit} for $\mpee(\osr)_{\rm sa}$: for every
$H\in \mpee(\osr)_{\rm sa}$ there is a $t>0$ such that $-t(I_n\otimes I_p)\leq H \leq t(I_n\otimes I_p)$ and 
$t(I_n\otimes I_p)+H\in\mpee(\osr)_+$ for all $t>0$ if and only if $H\in\mpee(\osr)_+$.

Lastly, there is an intimate relationship between the norm and the ordering: for every $Z\in\mpee(\osr)$,
\[
\|Z\|\;=\;\inf\left\{t>0\,:\, \left[
\begin{array}{cc}  t\,(I_n\otimes I_p) & Z \\ Z* & t\,(I_n\otimes I_p) \end{array}\right]
\in \mtwopee(\osr)_+
\right\}\,.
\]

\begin{definition} Assume that $\osr\subseteq\mn$ is a subspace that is closed under the
conjugate transpose $X\mapsto X^*$
and
contains the identity matrix, and let $\phi:\osr\rightarrow\mn$ be any linear transformation.
\begin{enumerate}
\item The \emph{norm} of $\phi$ is defined by $\|\phi\|=\max\{\|\phi(X)\|\,:\,X\in\osr,\;\|X\|=1\}$.
\item If $\osr=\mn$, then $\phi^k$ denotes $\phi\circ\dots\circ\phi$, the composition of $\phi$
with itself $k$ times.
\item For any $p\in\mathbb N$, $\phi^{(p)}$ denotes the linear transformation 
\[
\phi^{(p)}:\mpee(\osr)\rightarrow\mpee(\mn)\,,\quad\phi^{(p)}\left([X_{ij}]\right)\,=\,
[\phi(X_{ij})]\,.
\]
\item If $\phi^{(p)}$ maps $\mpee(\osr)_+$ into $\mpee(\mn)_+$, for every $p\in\mathbb N$, then $\phi$ is 
called a \emph{completely positive linear transformation}.
\item If $\phi$ is completely positive and if $\phi(I)=I$, then $\phi$ is called a \emph{ucp map} (unital completely positive).
\item If $\osr=\mn$ and if $\phi$ is a ucp map, then $\phi$ is called a \emph{conditional expectation} if $\phi^2=\phi$.
\end{enumerate}
\end{definition}

The following theorem captures a few of the most important features of completely positive linear transformations of
matrix spaces.

\begin{theorem}\label{cp thm} Assume that $\osr\subseteq\mn$ is a subspace that is closed under the conjugate transpose and
contains the identity matrix, and let $\phi:\osr\rightarrow\mn$ be a completely positive linear transformation.
\begin{enumerate}
\item\label{ext thm} {\rm (Arveson Extension Theorem)} There is a completely positive linear transformation 
$\Phi:\mn\rightarrow\mn$ such that $\Phi_{\vert\osr}=\phi$.
\item\label{kraus decom[} {\rm (Stinespring--Kraus--Choi Representation)} There are linearly independent matrices $V_1,\dots,V_r\in\mn$ such that
\begin{equation}\label{kraus}
\phi(X)\,=\,\sum_{j=1}^r V_j^*XV_j\,,\;\forall\,X\in\osr\,.
\end{equation}
\item\label{ucp proj}  If $\osr=\mn$ and if $\phi$ is a conditional expectation with range $\oss$, then
 \[
\phi(YZ)\,=\,\phi(Y\phi(Z))\,,\quad\forall\,Y\in\oss,\;Z\in \mn\,.
\]
\end{enumerate}
\end{theorem}

Proofs for the assertions in Theorem \ref{cp thm} are given, respectively, in Theorem 7.5, Theorem 4.1,  
and Theorem 15.2 of \cite{Paulsen-book}.

\section{An Ergodic Theorem}

The following result is special case of a theorem of Dunford \cite{dunford1943}.  

\begin{theorem}\label{ergodic thm}{\rm (Ergodic Theorem)} If $\omega:\mn\rightarrow\mn$ is a linear transformation
of norm $1$ and has $1$ as an eigenvalue, then
\begin{equation}\label{ergodic}
\lim_{m\rightarrow\infty}\,\frac{1}{m}\sum_{k=0}^{m-1} \omega^k
\end{equation}
exists and the limit $\Omega$ in {\rm \eqref{ergodic}}
is an idempotent linear transformation with range $\ker(\omega-{\rm id}_{\mn})$ and
kernel $\mbox{\rm ran}(\omega-{\rm id}_{\mn})$.
\end{theorem}

\begin{proof} If $X\in\ker(\omega-{\rm id}_{\mn})$, then $\omega^k(X)=X$ for every $k\in\mathbb N$ and so
$\frac{1}{m}\sum_{k=0}^{m-1} \omega^k(X)=X$ for every $m\in\mathbb N$. Thus, on the subspace
$\ker(\omega-{\rm id}_{\mn})$, the limit in \eqref{ergodic} exists and coincides with the identity on
$\ker(\omega-{\rm id}_{\mn})$.

Suppose that $Y=(\omega-{\rm id}_{\mn})(X)$, for some $X\in\mn$. Thus,
\[
\begin{array}{rcl} 
\left\|\displaystyle\frac{1}{m}\displaystyle\sum_{k=0}^{m-1}\omega^k(Y)\right\|
&=& 
\left\|\displaystyle\frac{1}{m}\displaystyle\sum_{k=0}^{m-1}\omega^k\left(\omega-{\rm id}_{\mn}\right)(X)\right\|  \\ && \\
&=& \left\|\displaystyle\frac{1}{m}\left(\omega^m-{\rm id}_{\mn}\right)(X)\right\| \\ && \\
&\leq& \frac{1}{m}\,\|\omega^m-{\rm id}_{\mn}\|\,\|X\| \\ && \\
&\leq& \frac{2}{m}\,\|X\|\,.
\end{array}
\]
Hence, on the subspace
$\mbox{\rm ran}(\omega-{\rm id}_{\mn})$, the limit in \eqref{ergodic} exists and coincides with the zero transformation on
$\mbox{\rm ran}(\omega-{\rm id}_{\mn})$. 

For every $m\in\mathbb N$, $\|\frac{1}{m}\omega^m\|\leq\frac{1}{m}\|\omega\|^m=\frac{1}{m}$, and so $\frac{1}{m}\omega^m\rightarrow0$.
If $J$ is the Jordan canonical form of $\omega$, then $\frac{1}{m}J^m\rightarrow 0$ as well. This is true for every Jordan block of $J$
and in particular for every $\ell\times\ell$ Jordan block $J_\ell(1)$ for the eigenvalue $1$ of $\omega$. But if $\ell>1$, then $\frac{1}{m}J_\ell(1)^m$
fails to converge to the zero matrix, and so it must be that $\ell=1$. This proves that 
\begin{equation}\label{simple eval}
\ker\left((\omega- {\rm id}_{\mn})^2\right)\,=\,\ker(\omega- {\rm id}_{\mn})\,.
\end{equation}

The Rank-Plus-Nullity Theorem asserts that the dimensions of $\ker(\omega-{\rm id}_{\mn})$ and $\mbox{\rm ran}(\omega-{\rm id}_{\mn})$
sum to $n^2=\mbox{\rm dim}\,\mn$. Equation \eqref{simple eval} shows that 
$\ker(\omega-{\rm id}_{\mn})$ and $\mbox{\rm ran}(\omega-{\rm id}_{\mn})$
have zero intersection. Hence, $\mn$ is an algebraic direct sum of $\ker(\omega-{\rm id}_{\mn})$ and $\mbox{\rm ran}(\omega-{\rm id}_{\mn})$,
which proves that the limit \eqref{ergodic} exists and that the limit $\Omega$ 
is an idempotent.
\end{proof}

\begin{corollary}\label{spectral cor} If $\omega:\mn\rightarrow\mn$ is a linear transformation such that
$\|\omega\|=1$, and if $\lambda$ is an eigenvalue of
$\omega$ such that $|\lambda|=1$, then $\lambda$ is a semisimple eigenvalue in the sense that
\[
\ker\left((\omega-\lambda\,{\rm id}_{\mn})^2\right)\,=\,\ker(\omega-\lambda\,{\rm id}_{\mn})\,.
\]
\end{corollary}

\begin{proof} Let $\omega'=\frac{1}{\lambda}\omega$ and apply Theorem \ref{ergodic thm}.
\end{proof}

Our main application of the Ergodic Theorem is:

\begin{corollary}\label{condition expectation cor} If $\omega:\mn\rightarrow\mn$ is a 
unital completely positive linear transformation, then
$\Omega=\displaystyle\lim_{m\rightarrow\infty}\,\frac{1}{m}\sum_{k=0}^{m-1} \omega^k$ 
is a conditional expectation with range $\{X\in\mn\,:\,\omega(X)=X\}$, the set of fixed points of $\omega$.
\end{corollary}

A second application of the Ergodic Theorem is drawn from quantum information theory \cite{kuperberg2003}.
 
\begin{corollary}\label{pf} If $\omega:\mn\rightarrow\mn$ is a ucp map, then there is a sequence $\{k_j\}_{j\in\mathbb N}$
and a conditional expectation $\Phi$ on $\mn$ such that
\[
\Phi\,=\,\lim_{j\rightarrow\infty}\omega^{k_j}\,.
\]
Moreover, $\Phi$ is the unique conditional expectation in the set of cluster points of the set $\{\omega^k\}_{k\in\mathbb N}$.
\end{corollary}

\begin{proof} Suppose that $\omega$ is in Jordan canonical form $J$. By Corollary \ref{spectral cor}, 
every eigenvalue $\lambda$ of $\omega$ of modulus $1$ is semisimple, which is to say that the size of  
every Jordan block of $\lambda$ in $J$ is $1\times 1$. 
Hence, we may choose any sequence $\{k_j\}_{j\in\mathbb N}$ so that
the eigenvalues of $J^{k_j}$ 
accumulate around $1$ and $0$ as $j\rightarrow\infty$,
thereby yielding a limiting
matrix that is idempotent. Clearly this is the only such idempotent cluster point
of $\{J^k\}_{k\in\mathbb N}$. Going back from the Jordan form $J$ to $\omega$,
one concludes that $\Omega$ is a idempotent, unital, and completely positive.
\end{proof}

\section{Completely Positive Isometries of $\mn$}
 
A special case of a theorem of Kadison \cite[Theorem 10]{kadison1951} is:

\begin{theorem}\label{kadison} {\rm (Kadison's Isometry Theorem)} 
If $\phi:\mn\rightarrow\mn$ is a unital completely positive linear transformation 
such that $\|\phi(X)\|=\|X\|$ for all $X\in \mn$, 
then there exists $U\in\un$ such that $\phi(X)=U^*XU$ for all $X\in \mn$.
\end{theorem}

\begin{proof} Assume that $\phi$ has a Stinespring--Kraus--Choi representation that is given by 
\[
\phi(X)\,=\,\sum_{i=1}^r V_j^*XV_j\,,\quad X\in\mn\,,
\]
for some linearly independent $V_1,\dots,V_r\in\mn$. Let $\{e_1,\dots,e_r\}$ be the standard orthonormal
basis for $\mathbb C^r$ and consider the function $V:\mathbb C^n\rightarrow\mathbb C^n\otimes\mathbb C^r$
for which
\[
V\xi\,=\,\sum_{i=1}^r V_i\xi\otimes e_i\,,\quad\xi\in\mathbb C^n\,.
\]
Define an injective unital homomorphism $\pi:\mn\rightarrow\mn\otimes\mr$ by $\pi(X)=X\otimes I_r$. Thus,
\begin{equation}\label{e:skc}
\phi(X)\,=\,V^*\pi(X) V\,=\,\sum_{i=1}^r V_j^*XV_j\,,\quad X\in\mn\,.
\end{equation}
Furthermore, because $V_1,\dots,V_r\in\mn$ are linearly independent,
\begin{equation}\label{minimal}
\mbox{Span}\,\{\pi(X)V\xi\,|\,X\in\mn,\;\xi\in\mathbb C^n\}\,=\,\mathbb C^n\otimes\mathbb C^r\,.
\end{equation}

The linear map $\phi$ is an isometry of a finite-dimensional space; thus, $\phi$
has an isometric inverse. Therefore, if 
$W\in\un$, then $\phi(W)$ is the midpoint between $X,Y\in\mbox{\rm Ball}\left(\mn\right)$
only if $W$ is the midpoint between $\phi^{-1}(X),\phi^{-1}(Y)\in\mbox{\rm Ball}\left(\mn\right)$, which is possible
only if $\phi^{-1}(X)=\phi^{-1}(Y)=W$ because unitary matrices are extreme points of 
$\mbox{\rm Ball}\left(\mn\right)$. Thus, $\phi(W)$ is an extreme point of 
$\mbox{\rm Ball}\left(\mn\right)$, which is to say that $\phi(W)\in\un$ for all $W\in\un$.

Decompose $\mathbb C^n\otimes\mathbb C^r$ as $\mbox{ran}\,V\oplus (\mbox{ran}\,V)^\bot$ and choose $W\in\un$.
With respect to this decomposition of $\mathbb C^n\otimes\mathbb C^r$,
the unitary matrix $\pi(W)$ has the form
\[
\pi(W)\,=\,\left[ \begin{array}{cc} \phi(W) & Z_{12} \\ Z_{21} & Z_{22} \end{array} \right]\,.
\]
Since
\[
\left[\begin{array}{cc} I_n&0\\ 0&I_{(r-1)n}\end{array}\right]
\,=\,\pi(W)^*\pi(W)\,=\,\left[ \begin{array}{cc} \phi(W)^*\phi(W)+Z_{21}^*Z_{21} & * \\ * & *\end{array}\right]\,,
\]
we have $Z_{21}^*Z_{21}=I_n-\phi(W)^*\phi(W)=0$ (as $\phi(W)$ is unitary). Thus, $Z_{21}=0$. Likewise, from $\pi(W)\pi(W)^*=I_{nr}$,
we deduce that $Z_{12}=0$.
Therefore, the off-diagonal blocks of $\pi(W)$ must be zero. This is true for every $W\in\un$,
and because $\un$ spans $\mn$, it is also true that
\[
\pi(X)\,=\,\left[ \begin{array}{cc} \phi(X) & 0 \\ 0 & *\end{array}\right]\,,
\]
for every $X\in\mn$. That is, the subspace $\mbox{ran}\,V$ is $\pi(X)$-invariant, for every $X\in \mn$. But
in light of \eqref{minimal}, this implies that the range of $V$ is $\mathbb C^n\otimes \mathbb C^r$, which is
possible only if $r=1$. Thus, $V_1$ is unitary and taking $U=V_1$ completes the
proof of the theorem.
\end{proof}
 
\section{Fixed Points}

The deepest aspect of Arveson's criterion for unitary similarity is the following theorem concerning
the set $\{X\in\mn\,:\,\omega(X)=X\}$ of fixed points of a unital completely positive linear transformation $\omega$
of $\mn$.

\begin{theorem}\label{boundary thm} {\rm (Boundary Theorem)} 
If $A\in\mn$ is irreducible and if $\omega:\mn\rightarrow\mn$
is a unital completely positive linear transformation such that $\omega(A)=A$, then 
$\omega(X)=X$ for every $X\in\mn$.
\end{theorem}

\begin{proof}
Let $\osr=\mbox{Span}\,\{I,A,A^*\}$ so that $\mn$ is the algebra generated by $\osr$
and $\omega_{\vert\osr}={\rm id}_\osr$. Let $\oss=\{X\in\mn\,:\,\omega(X)=X\}$, which is a
unital subspace of $\mn$ that contains the identity matrix and is closed under the involution $Z\mapsto Z^*$.
Because $\oss\supseteq\osr$, the algebra generated by $\oss$ is $\mn$.

The Ergodic Theorem asserts that $\Omega=\displaystyle\lim_{m\rightarrow\infty}\,\frac{1}{m}\sum_{k=0}^{m-1} \omega^k$
is a conditional expectation that maps $\mn$ onto the fixed point space $\oss$.
Thus, by the Choi--Effros Theorem \cite{choi--effros1977}, \cite[Theorem 15.2]{Paulsen-book},
 the linear space $\oss$ is an algebra
under the product $\odot$ defined by
\begin{equation}\label{ce product}
X\odot Y\,=\,\Omega(XY)\,,\quad X,Y\in\oss\,.
\end{equation}
If $Y\in\oss$ and $Z\in\mn$, 
then by Theorem \ref{cp thm}\eqref{ucp proj},
\[
\Omega (YZ)\;=\;\Omega(Y\Omega(Z))\;=\;Y\odot\Omega(Z)\;=\;\Omega(Y)\odot\Omega(Z)\,.
\]
Likewise, if $Y_1,Y_2\in \oss$ and $Z\in\mn$, then 
\[
\Omega((Y_1Y_2)Z)\;=\;\Omega(Y_1\Omega(Y_2Z))\;=\;Y_1\odot \Omega(Y_2Z)\;=\;\left(\Omega(Y_1)\odot\Omega(Y_2)\right)\odot\Omega(Z)\,.
\]
By induction, if $\mathfrak a$ is any word in $2q$ noncommuting variables, and if $Y_1,\dots,Y_q\in\oss$ and $Z\in\mn$, then
\[
\Omega\left( \mathfrak a(Y_1,\dots,Y_q,Y_1^*,\dots,Y_q^*)Z\right)\;=\;
\left(\mathfrak a_{\odot}(\Omega(Y_1),\dots,\Omega(Y_q),\Omega(Y_1)^*,\dots,\Omega(Y_q)^*)\right)\odot\Omega(Z)\,,
\]
where $\mathfrak a_{\odot}(\Omega(Y_1),\dots,\Omega(Y_q),\phi(Y_1)^*,\dots,\Omega(Y_q)^*)$ 
denotes the $\odot$-product of the letters of the word $\mathfrak a$. Because the algebra generated by $\oss$, namely $\mn$,
 is given by
linear combinations of elements of the form
$ \mathfrak a(Y_1,\dots,Y_q,Y_1^*,\dots,Y_q^*)$ for various positive integers $q$, words $\mathfrak a$, and elements $Y_j\in\oss$,
the linear transformation $\Omega$ satisfies
\[
\Omega(WZ)\;=\;\Omega(W)\odot\Omega(Z)\,,\;\forall\,W,Z\in \mn\,.
\]
That is, $\Omega$ is a homomorphism of the associative algebra $\mn$ onto the associative algebra $\oss$ with product $\odot$.
Because $\mn$ has no nontrivial ideals and $\oss\neq\{0\}$, 
$\Omega$ must in fact be an isomorphism. Thus, $\ker\Omega=\{0\}$, which implies that
the idempotent $\Omega$ is the identity transformation. Therefore, the range of $\Omega$, namely the fixed point set $\oss$,
is all of $\mn$.
\end{proof}

\begin{corollary} If $\omega$ is a unital completely positive linear transformation of $\mn$
for which $\ker(\omega-{\rm id}_{\mn})\cap\mathcal O_{\rm irr}\neq\emptyset$, then $\omega$
is the identity transformation.
\end{corollary}

\begin{corollary}\label{choquet} {\rm (Noncommutative Choquet Theorem)} If $A\in\mn$ is irreducible
and $\osr=\mbox{\rm Span}\,\{I,A,A^*\}$, then the unital completely positive linear transformation
$\iota:\osr\rightarrow\mn$ defined by $\iota(X)=X$, for $X\in\osr$, has a unique 
completely positive extension to $\mn$.
\end{corollary}

\section{Proof of Theorem \ref{main result}}
 
If $A,B\in\mn$ are unitarily similar, 
then a straightforward calculation verifies that 
$\|A\otimes H+I\otimes K\|=\|B\otimes H+I\otimes K\|$ for all $H,K\in\mn$.

Conversely, assume that
$A,B\in\mn$, $A\in\mathcal O_{\rm irr}$, and $\|A\otimes H+I\otimes K\|=\|B\otimes H+I\otimes K\|$, for all $H,K\in\mn$.
Define a linear map $\phi_0:\mbox{Span}\{I,A\}\rightarrow\mn$
by 
\[
\phi_0(\alpha_0I +\alpha_1 A)\;=\;\alpha_0I+\alpha_1B\,,\;\forall\,\alpha_0,\alpha_1\in\mathbb C\,.
\]
Because $\|A\otimes H+I\otimes K\|=\|B\otimes H+I\otimes K\|$ for all $H,K\in\mn$, 
the linear map 
\[
\phi_0^{(n)}:\mbox{Span}\{I,A\}\otimes\mn\rightarrow\mn\otimes\mn\,,
\]
in which
\[
\phi_0^{(n)}\left([X_{st}]_{1\leq s,t\leq n}\right)\,=\,[\phi_0(X_{st})]_{1\leq s,t\leq n}\,,
\]
is an isometry. 

Let $\osr=\mbox{span}\,\{I,A,A^*\}$. 
By \cite[Proposition 3.5]{Paulsen-book}, the unital  linear transformation
$\phi:\osr\rightarrow\mn$ defined by
\[
\phi(\alpha I+\beta A+\gamma  A^*)\;=\;\alpha I+\beta \phi_0(A)+\gamma\phi_0(A)^* 
\]
is completely positive and satisfies $\phi(A)=B$. Therefore, by Theorem \ref{cp thm}{\eqref{ext thm}, there is a completely positive extension of
$\phi$ from $\osr$ to $\mn$; without loss of generality, let $\phi$ denote the extended completely positive transformation
of $\mn$.
By similar reasoning, there is a unital completely positive linear transformation $\psi:\mn\rightarrow\mn$ such that $\psi(B)=A$.
Hence, $\omega=\psi\circ\phi$ is a unital
completely positive linear transformation of $\mn$ with $\omega(A)=A$.
By the Boundary Theorem (Theorem \ref{boundary thm}), $\omega=\psi\circ\phi$ is the identity transformation, and so
\[
\|X\|\;=\;\|\psi\left(\phi(X)\right)\|\;\leq\;\|\phi(X)\|\;\leq\;\|X\|
\]
for every $X\in\mn$.
That is, $\phi:\mn\rightarrow \mn$ is a unital completely positive isometry. 
Therefore, by Kadison's Isometry Theorem (Theorem \ref{kadison}), there is a $U\in\un$ such that
$\phi(X)=U^*XU$ for every $X\in\mn$. Hence, $B=U^*AU$.

\section{Discussion}

Any proof of Arveson's criterion for unitary similarity likely requires the Boundary Theorem (Theorem \ref{boundary thm}). 
If one compares the proof of Specht's Theorem, as given by Kaplansky in \cite[Theorem 63]{Kaplansky-linear-algebra-book},
with the proof of the Boundary Theorem herein, it is clear that properties of matrix rings have a crucial role in arriving 
at these results, even if the statements of the results are concerned only with single matrices and the
proofs, for the most part, involve only
linear spaces of matrices.

Our proof of the Boundary Theorem is different from Arveson's (and from Davidson's \cite{davidson1981})
in that it is based on methods that are used in the study of the noncommutative \v Silov boundary, which was
introduced by Arveson in \cite{arveson1969} and developed further by
Hamana \cite{hamana1979b} and Blecher \cite{blecher2001}. In contrast, Arveson and Davidson approach the theorem from the perspective
of the noncommutative Choquet boundary\footnote{The Choquet boundary of a linear space $\ose$ of continuous
complex-valued functions on a compact Hausdorff space $X$---where $\ose$ separates the points of $X$, contains the constant functions, and is closed under
complex conjugation---is the set of all $x_0\in X$ for which the positive linear functional $f\mapsto f(x_0)$, $f\in\ose$, has a unique
extension to a positive linear functional on the space of all continuous functions $g:X\rightarrow\mathbb C$. Corollary \ref{choquet}
is exactly this idea, but in a noncommutative environment in which $\osr$ plays the role of $\ose$.}. These noncommutative
\v Silov and Choquet boundaries are used by Arveson \cite{arveson2010} to classify,
up to complete order isomorphism, all subspaces of matrices that contain
the identity matrix and are closed under the conjugate transpose. Such a classification is indeed a broader, more sophisticated 
form of the main theorem (on unitary similarity) of the present paper, yet is still within the scope and interest
of core linear algebra.

\section{Acknowledgement}

I thank Vladimir Sergeichuk for helpful discussions on the problem of unitary similarity
and for his encouragement to write this account of Arveson's
criterion, and Roger Horn for very useful editorial suggestions. The original draft of
this paper was written at Institut Mittag-Leffler (Djursholm, Sweden) in November 2010. This work is
supported in part by NSERC (Canada).

 
\bibliographystyle{amsplain2}
\bibliography{arv}

\end{document}